\documentclass{amsart}

\usepackage[cmyk]{xcolor}
\usepackage{ulem}
\usepackage{listings}
\usepackage{amsmath}
\usepackage{amsbsy}
\usepackage{amssymb}
\usepackage{amsthm}
\usepackage{amscd}
\usepackage{mathrsfs}
\usepackage[abbrev]{amsrefs}
\usepackage[all]{xy}

\newcommand{\bbz}{{\mathbb Z}}

\newcommand{\lan}{\langle}
\newcommand{\ran}{\rangle}

\newtheorem{thm}{Theorem}[section]
\newtheorem{lem}[thm]{Lemma}

\newtheorem{prop}[thm]{Proposition}

\newtheorem*{ack*}{Acknowledgment}

\newtheorem*{note*}{Notation and conventions}

\DeclareFontEncoding{OT2}{}{}
\DeclareFontSubstitution{OT2}{cmr}{m}{n}

\DeclareFontFamily{OT2}{cmr}{\hyphenchar\font45 }
\DeclareFontShape{OT2}{cmr}{m}{n}{%
   <5><6><7><8><9>gen*wncyr%
   <10><10.95><12><14.4><17.28><20.74><24.88>wncyr10}{}
\DeclareFontShape{OT2}{cmr}{b}{n}{%
   <5><6><7><8><9>gen*wncyb%
   <10><10.95><12><14.4><17.28><20.74><24.88>wncyb10}{}

\DeclareMathAlphabet{\mathcyr}{OT2}{cmr}{m}{n}
\DeclareMathAlphabet{\mathcyb}{OT2}{cmr}{b}{n}
\SetMathAlphabet{\mathcyr}{bold}{OT2}{cmr}{b}{n}

\begin{document}

\title{Automorphism group of the holomorph of a cyclic group}
\author{Kazuki Sato}
\begin{abstract}
We show that the holomorph of a cyclic group of order $n$ is isomorphic to its own automorphism group when $n$ is twice of a power of an odd prime.
\end{abstract}
\date{\today}
\address{National Institute of Technology, Ichinoseki College, Japan}
\email{kazuki-s@ichinoseki.ac.jp}
\maketitle

\section{Introduction}
Let $C_n$ be a cyclic group of order $n$.
In this paper we prove the following.

\begin{thm}[{Theorem \ref{main}}]\label{intromain}
Let $p$ be an odd prime and $e>0$ a positive integer.
Assume $n=2p^e$.  
Let $G=\operatorname{Hol}(C_n)$ be the holomorph of $C_n$.
Then there exists an isomorphism $G \cong \operatorname{Aut}(G)$. 
\end{thm}

It is known that the automorphism group of the dihedral group $D_{2n}$ of order $2n$ is isomorphic to the holomorph $\operatorname{Hol}(C_n)$ of $C_n$ \cite{walls}.  
It is also known that the holomorph $\operatorname{Hol}(C_n)$ of $C_n$ is a complete group when $n$ is odd \cite{miller}.
In particular, $\operatorname{Hol}(C_n)$ is isomorphic to $\operatorname{Aut}(\operatorname{Hol}(C_n))$ when $n$ is odd.
But when $n$ is even, the center $Z(\operatorname{Hol}(C_n))$ is nontrivial (see Lemma \ref{center}), so $\operatorname{Hol}(C_n)$ is not complete.
Despite that fact, Theorem \ref{intromain} asserts that two groups $\operatorname{Hol}(C_n)$ and 
$\operatorname{Aut}(\operatorname{Hol}(C_n))$ are isomorphic as abstract groups when $n=2p^e$.
We note that the case $e=1$, that is, $n=2p$ in Theorem \ref{intromain}
is known to be true \cite[Theorem 2.10]{svb}.

In section $2$, we recall preliminary results from elementary group theory and number theory.
We show our main theorem in section $3$.

\section{Preliminaries}

\subsection{Semi-direct product}
We will review briefly the notion of a semi-direct product.
Given two groups $A$ and $B$, and a homomorphism $\theta : B \to \operatorname{Aut}(A)$,
the semi-direct product gives a group structure to the set $G = A \times B$ by
\[ (a_1, b_1) \cdot (a_2 , b_2) =(a_1 \theta(b_1)(a_2), b_1 b_2).\]
We will denote the semi-direct product of $A$ by $B$ with $\theta : B \to \operatorname{Aut}(A)$ by 
$A \rtimes_\theta B$,
where $\theta$ may be omitted if it is clear from the context.
We identify an element $a \in A$ with $(a,1)$ and $b \in B$ with $(1,b)$.
With these identifications we can see that all the elements of $A \rtimes_\theta B$ can be written uniquely in the form $ab$ with $a \in A$ and $b \in B$.
Multiplication of elements from $A$ and $B$ is determined by the relation $bab^{-1}=\theta(b)(a)$.
From these identifications we can write $A \rtimes_\theta B=AB$ with $A \triangleleft A \rtimes_\theta B$ and $A \cap B =1$.
In particular $(A \rtimes_\theta B)/A \cong B$.
If $B= \operatorname{Aut}(A)$ and $\theta$ is the identity map, then we call $A\rtimes \operatorname{Aut}(A)$ the holomorph of $A$ and write it as $\operatorname{Hol}(A)$.

We use the following theorem in section 3.
\begin{thm}[{\cite[Theorem 3.35]{Isaacs}}]\label{Is335}
Let $G$ and $G_0$ be groups, 
and let $N \triangleleft G$ and $N_0 \triangleleft G_0$,
where $G/N$ and $G_0/N_0$ are cyclic of finite order $m$.
Suppose that $y : N \to N_0$ is an isomorphism,
and let $gN$ and $g_0N_0$ be generating cosets of $G/N$ and $G_0 / N_0$ respectively.
Assume that
\[ y(g^m) = (g_0)^m \ \ \ \textrm{and} \ \ \ y(gxg^{-1}) = g_0y(x)g_0^{-1}\]
for all elements $x \in N$.
Then there is a unique isomorphism $\alpha : G \to G_0$ that extends the given isomorphism $y:N \to N_0$ and that satisfies $\alpha(g)=g_0$.
\end{thm}

\subsection{Dihedral Groups}

Let $n$ be a positive integer.
In this paper we let $C_n$ denote the cyclic group of order $n$ and $D_{2n}$ the dihedral group of order $2n$.
As mentioned in the Introduction, it is known that the automorphism group $\operatorname{Aut}(D_{2n})$ of $D_{2n}$ is isomorphic to the holomorph $\operatorname{Hol}(C_n)$ of $C_n$.

\begin{lem}\label{center}
Let $n$ be an even integer.
Let $G=\operatorname{Hol}(C_n)$ be the holomorph of $C_n$.
Then the center $Z(G)$ of $G$ is equal to $\{1, z\}$, where $z \in C_n$ is the unique involution.
\end{lem}
   
\begin{proof}   
Since $y(z) \in C_n$ is also an involution for any $y \in \operatorname{Aut}(C_n)$, we have $yzy^{-1}=y(z)=z$.
Hence $\{ 1, z \} \subset Z(G)$.
We have to show the reverse inclusion.
   
Let $xy \in Z(G)$, where $x \in C_n, y \in \operatorname{Aut}(C_n)$.
Then we have $(xy)x'=x'(xy)$ for any $x' \in C_n$.
Note that $yx'=y(x')y$, so $xy(x')y=x(yx')=(xy)x'=x'(xy)=xx'y$ for any $x' \in C_n$.
This implies that $y(x')=x'$ for any $x' \in C_n$, hence $y=1 \in \operatorname{Aut}(C_n)$.
Let $\iota \in \operatorname{Aut}(C_n)$ be the automorphism given by $\iota(x)=x^{-1}$.
Since $x \in Z(G)$, we have $x=\iota x \iota^{-1} =\iota(x)=x^{-1}$.
Therefore $x^2=1$, that is, $x=1$ or $x=z$ is the involution.
This shows that $Z(G) \subset \{1, z\}$.
\end{proof}

If $G$ is an arbitrary group, then $G/Z(G)$ is naturally isomorphic to the group of inner automorphisms $\operatorname{Inn}(G)$.
If $Z(G)=1$, therefore, $G \cong \operatorname{Inn}(G)$, and so we can identify $G$ with $\operatorname{Inn}(G)$ via the natural isomorphism, and this embeds $G$ as a subgrouop of the automorphism group $\operatorname{Aut}(G)$.
We say that $G$ is complete if $Z(G)=1$ and $\operatorname{Aut}(G)=\operatorname{Inn}(G)$.
It is known that $\operatorname{Aut}(D_{2n})\cong \operatorname{Hol}(C_n)$ is complete when $n$ is odd.
But when $n$ is even, $\operatorname{Aut}(D_{2n})$ can never be complete.

\subsection{Elementary number theory}
\begin{lem}\label{powermodp}
Let $k$ be an integer and $p$ a prime.
Assume that $k \equiv 1 \mod p$.
Then $k^{p^{e-1}} \equiv 1 \mod p^e$ for any $e \geq 1$.
\end{lem}

\begin{proof}
By induction on $e$.
Assume that $k^{p^{e-1}} \equiv 1 \mod p^e$ for some $e \geq 1$.
Then $k^{p^{e-1}} = 1+ rp^e$ for some integer $r$.
Hence we have
\[ k^{p^e}=(1+rp^e)^p = 1+ \sum_{i=1}^{p-1} \binom{p}{i} r^ip^{ei} + r^p p^{ep}.\]
For each $1 \leq i \leq p-1$, since $\binom{p}{i}$ is divisible by $p$, we see that $\binom{p}{i}p^{ei} \equiv 0 \mod p^{e+1}$.
Clearly $p^{ep} \equiv 0 \mod p^{e+1}$.
This shows that $k^{p^e} \equiv 1 \mod p^{e+1}$.
\end{proof}

\begin{lem}\label{elnum}
Let $p$ be an odd prime and $e>0$.
Let $k$ be an integer coprime to $n=2p^e$.
Assume that the order of $k \in (\bbz/n\bbz)^*$ is $\phi(n)=p^{e-1}(p-1)$, where $\phi$ is Euler's totient funciton.
Then
\begin{enumerate}
\item[(i)] $k-1$ is not divisible by $p$, 
\item[(ii)] the greatest common divisor of $n$ and $k-1$ is $(n,k-1)=2$,
\item[(iii)] $1+k+k^2+\dots + k^{\phi(n)-1} \equiv 0 \mod n$.
\end{enumerate}
\end{lem}

\begin{proof}
(i)
Assume $k \equiv 1 \mod p$.
Then $k^{p^{e-1}} \equiv 1 \mod p^e$ from Lemma \ref{powermodp}.
Since $k$ is odd, we also have $k^{p^{e-1}} \equiv 1 \mod 2$.
Hence $k^{p^{e-1}} \equiv 1 \mod n$, so
the order of $k \in (\bbz/n\bbz)^*$ divides $p^{e-1}<p^{e-1}(p-1)=\phi(n)$, contradiction.

(ii)
Since $n$ and $k-1$ are even, $2|(n,k-1)$.
On the other hand, $(n,k-1)$ is not divisible by $p$ by (i).
This shows (ii).

(iii)
See \cite[Theorem 2.5]{svb}.
\end{proof}

\section{Main theorem}
In this section, we fix the notation as follows.
Let $e>0$ be an integer and $p$ an odd prime.
Put $n:= 2p^e$.
Then the automorphism group $\operatorname{Aut}(C_n)$ is a cyclic group of order $\phi(n)=p^{e-1}(p-1)$.
We fix generators $x \in C_n$ and $y \in \operatorname{Aut}(C_n)$.
Assume that $y(x)=x^k$, where $k$ is an integer coprime to $n$ with $1<k<n$.
In this situation, there exists an isomorphism $\operatorname{Aut}(C_n) \cong (\bbz / n\bbz)^*$ that maps $y$ to $k$.
We note that the order of $k \in (\bbz/n\bbz)^*$ is $\phi(n)$. 
Let $G:= \operatorname{Hol}(C_n) = C_n \rtimes \operatorname{Aut}(C_n)$ be the holomorph of $C_n$.
Then $G$ is generated by $x$ and $y$, and the elements of $G$ can be written in the form $x^a y^b$.
We see that $x^a y^b=x^c y^d$ if and only if $a \equiv c \mod n$ and $b \equiv d \mod \phi(n)$.
We have the following relations in $G$;
\[ x^n=1, \ y^{\phi(n)}=1, \ yxy^{-1}=y(x)=x^k.\]

\begin{lem}\label{lemma31}
$y^b x^a y^{-b}=x^{ak^b}$ holds for any integer $a$ and nonnegative integer $b$.
\end{lem}

\begin{proof}
By induction on $b$.
The case $b=0$ is trivial.
Assume that $b>0$ and that $y^{b-1}x^a y^{-(b-1)}=x^{ak^{b-1}}$ for any $a$.
Then we have
$y^bx^a y^{-b}=y^{b-1}yx^{a}y^{-1}y^{-(b-1)}=
y^{b-1}x^{ak}y^{-(b-1)}
=x^{akk^{b-1}}
=x^{ak^b}$.
Here, for the second equality, we used the relation in $G$ and for the third equality, we used the inductive hypothesis. 
\end{proof}

\begin{lem}\label{lemma32}
$(x^ay^b)^m=x^{a(1+k^b+\dots + k^{b(m-1)})}y^{bm}$ 
holds for any integer $a$, nonnegative integer $b$ and positive integer $m$.
\end{lem}

\begin{proof}
By induction on $m$.
The case $m=1$ is trivial.
Assume that $m>1$.
Then
\begin{align*}
(x^ay^b)^m &= (x^ay^b)^{m-1} x^a y^b \\
&=x^{a(1+k^b+\dots + k^{b(m-2)})}y^{b(m-1)}x^ay^b \\
&= x^{a(1+k^b+\dots + k^{b(m-2)})}y^{b(m-1)}x^ay^{-b(m-1)}y^{bm} \\
&=x^{a(1+k^b+\dots + k^{b(m-2)})}x^{ak^{b(m-1)}}y^{bm} \\
&=x^{a(1+k^b+\dots + k^{b(m-1)})}y^{bm}.
\end{align*}
Here, for the second equality, we used the inductive hypothesis and for the fourth equality, we used Lemma \ref{lemma31}.
\end{proof}

\begin{lem}\label{x}
Let $\alpha \in \operatorname{Aut}(G)$.
Then $\alpha(x) = x^a$ for some integer $a$ coprime to $n$.
Hence there exists an integer $j$ such that $\alpha(x)=y^j(x)$.
\end{lem}

\begin{proof}
The last statement is clear from the first one.
Let $\alpha \in \operatorname{Aut}(G)$.
Assume $\alpha(x)=x^ay^b, \alpha(y)=x^c y^d$ for some integers $a,b,c,d$, with $0 \leq b < \phi(n)$.
First, we will show that $b=0$.
We must have $\alpha(y)\alpha(x) \alpha(y^{-1})=\alpha(x^k)$.
It follows that 
\[(x^cy^d)(x^ay^b)(x^c y^d)^{-1}=(x^a y^b)^k, \]
which considered modulo $\lan x \ran$ implies that 
$y^b=y^{bk}$, and thus consequently that $y^{bk-b}=1$.
So $bk \equiv b \mod \phi(n)$.
Assume $0<b<\phi(n)$.
We note that $k^b -1 \neq 0$, since $1<k <n$. 
By Lemma \ref{lemma32}, we see that
\[(x^ay^b)^{k-1}=x^{a(1+k^b+ \dots + k^{b(k-2)})}y^{b(k-1)}=x^{a(k^{b(k-1)}-1)/(k^b -1)}.\]
Since the order of $x^ay^b=\alpha(x)$ is $n$, 
the order of $(x^ay^b)^{k-1}$ is $n/(n,k-1)=n/2$
from Lemma \ref{elnum}, (ii).
On the other hand, the order of 
$x^{a(k^{b(k-1)}-1)/(k^b -1)}$ is $n/\left(n, a\dfrac{k^{b(k-1)}-1}{k^b-1} \right)$.
Hence
\begin{equation}\label{eq1}
\left(n, a\dfrac{k^{b(k-1)}-1}{k^b-1} \right)=2.
\end{equation}
Recall that $b(k-1) \equiv 0 \mod \phi(n)$.
Hence
\begin{equation}\label{eq2}
k^{b(k-1)} \equiv 1 \mod n.
\end{equation}
(\ref{eq1}) and (\ref{eq2}) imply that 
$k^b -1 \equiv 0 \mod p^e$.
Since $k$ is odd, 
we also have $k^b-1 \equiv 0 \mod 2$.
Therefore we can conclude that $k^b-1
\equiv 0 \mod n$.
This contradicts the assumptions that $k \in (\bbz/n\bbz)^*$ has order $\phi(n)$ and that $0<b<\phi(n)$.
Hence $b=0$ and $\alpha(x)=x^a$.
Since $\alpha(x)=x^a$ has order $n$,
$a$ is coprime to $n$ and this completes the proof.
\end{proof}

\begin{lem}\label{y}
   Let $\alpha \in \operatorname{Aut}(G)$.
   Then $\alpha(y) = x^cy$ for some integer $c$.
\end{lem}

\begin{proof}
Let $\alpha \in \operatorname{Aut}(G)$.
By Lemma \ref{x}, we can assume that $\alpha(x)=x^a, \alpha(y)=x^c y^d$ for some integers $a,c,d$ with $a$ coprime to $n$ and $0 \leq d < \phi(n)$.
As above we must have
\[(x^cy^d)(x^a)(x^c y^d)^{-1}=(x^a)^k. \]
This implies $x^{ak^d}=x^{ak}$,
so $ak^d \equiv ak \mod n$.
Since $ak$ is coprime to $n$,
we have $k^{d-1} \equiv 1 \mod n$.
Hence $d-1 \equiv 0 \mod \phi(n)$.
Therefore $y^d=y$, and we have $\alpha(y)=x^c y^d = x^c y$.
\end{proof}

We define a map $\psi : \operatorname{Aut}(G) \to G$ as follows.
Let $\alpha \in \operatorname{Aut}(G)$.
By Lemmas \ref{x} and \ref{y}, 
there exist integers $j$ and $c$ such that $\alpha(x)=y^j(x), \alpha(y)=x^c y$.
We note that $y^j$ and $x^c$ are uniquely determined from $\alpha$.
We define
\[\psi(\alpha)=x^c y^j .\]

\begin{prop}\label{gphom}
The map $\psi : \operatorname{Aut}(G) \to G$ is a group homomorphism.
\end{prop}

\begin{proof}
Let $\alpha, \beta \in \operatorname{Aut}(G)$.
Assume $\psi(\alpha)=x^c y^j$, $\psi(\beta)=x^{d}y^{m}$.
Then we have
$(\alpha\beta)(x)=\alpha(\beta(x))=\alpha(y^m(x))
=y^j(y^m(x))
=y^{j+m}(x)
$
and
$(\alpha\beta)(y)=\alpha(\beta(y))
=\alpha(x^{d}y)
=\alpha(x)^{d}\alpha(y)
=y^j(x^d)x^cy
$.
Hence
$\psi(\alpha \beta)=y^j(x^d)x^cy^{j+m}$.
On the other hand, we have
$\psi(\alpha) \psi(\beta)=x^cy^jx^{d}y^{m}
=x^cy^j(x^d)y^{j+m}
=y^j(x^d)x^c y^{j+m}
$.
Therefore $\psi(\alpha) \psi(\beta)
=\psi(\alpha \beta)$.
\end{proof}

\begin{prop}\label{inj}
The map $\psi : \operatorname{Aut}(G) \to G$ is injective.
\end{prop}

\begin{proof}
Let $\alpha \in \operatorname{Aut}(G)$.
Assume $\psi(\alpha)=x^c y^j=1$.
Then $x^c=1$ and $y^j=1$, so
we have $\alpha(x)=y^j(x)=x$ and 
$\alpha(y)=x^cy=y$.
Since $G$ is generated by $x$ and $y$, we see that $\alpha$ is the identity map on $G$.
\end{proof}

\begin{prop}\label{surj}
The map $\psi : \operatorname{Aut}(G) \to G$ is surjective.
\end{prop}

\begin{proof}
Let $z \in G$.
Then $z=x^c y^j$ for some $0 \leq c < n$ and $0 \leq j < \phi(n)$.
We note that $yC_n=(x^cy)C_n$ is a generating coset of $G/C_n$.
We have
$(x^cy)^{\phi(n)}=x^{c(1+k+\dots +k^{\phi(n)-1})}
y^{\phi(n)}
=1$ by Lemmas \ref{lemma32} and \ref{elnum}, (iii).
Furthermore, 
$y^j(yxy^{-1})=y^j(y(x))=y^{j+1}(x)=x^c y(y^{j}(x))x^{-c}=(x^cy)y^j(x)(x^cy)^{-1}$.
Taking $i$-th power of both sides, we have
$y^j(yx^iy^{-1})=(x^cy)y^j(x^i)(x^cy)^{-1}$ for all integers $i$.
By Theorem \ref{Is335}, there exists a unique isomorphism $\alpha :G \to G$ that extends the isomorphism $y^j : \lan x\ran \to \lan x \ran$ so that $\alpha(y)=x^cy$.
Then $\psi(\alpha)=x^cy^j = z$ by definiton of $\psi$.
\end{proof}

\begin{thm}\label{main}
The map $\psi : \operatorname{Aut}(G) \to G$ is an isomorphism.
\end{thm}
\begin{proof}
This follows from Propositions \ref{gphom}, \ref{inj} and \ref{surj}.
\end{proof}

\end{document}